\documentclass[a4paper,12pt]{amsart}

\usepackage[utf8]{inputenc}
\usepackage{amsmath, amsthm, amsfonts, amssymb}
\usepackage[foot]{amsaddr}
\usepackage{stmaryrd}	%for llbracket
\usepackage{graphicx}
\usepackage{xcolor}

\usepackage{hyperref} %internal references
\usepackage[mathlines, pagewise]{lineno} %for adding line numbering on the margin, useful when submitting for review
\newcommand{\eps}{\varepsilon}
\newcommand{\err}{{\rm err}}

\newcommand{\N}{\mathbb{N}}
\newcommand{\Z}{\mathbb{Z}}
\newcommand{\R}{\mathbb{R}}
\newcommand{\C}{\mathbb{C}}

\newcommand{\be}{\begin{equation}}
\newcommand{\ee}{\end{equation}}

\newtheorem{theorem}{Theorem}
\newtheorem{lemma}[theorem]{Lemma}
\newtheorem{corollary}[theorem]{Corollary}
\newtheorem{prop}[theorem]{Proposition}
\theoremstyle{definition}
\newtheorem{remark}[theorem]{Remark}

\newtheorem{alg}[theorem]{Algorithm}

\title[Tractability of sampling recovery]{Tractability of %$L_p$-approximation 
sampling recovery
on unweighted function classes}

\author{
David Krieg %$^\star$
}

\address{Institut f\"ur Analysis, 
Johannes Kepler Universit\"at Linz, Austria.}
\email{david.krieg@jku.at}

\begin{document}

\begin{abstract}
It is well-known that the problem of sampling recovery in the $L_2$-norm
on unweighted Korobov spaces (Sobolev spaces with mixed smoothness) 
as well as classical smoothness classes such as 
Hölder classes suffers from the curse of dimensionality.
We show that the problem is tractable for those classes if 
they are intersected with the Wiener algebra of functions with summable Fourier coefficients.
In fact, this is a relatively simple implication of powerful results 
from the theory of compressed sensing.
%by Rauhut and Ward [Appl.\ Comput.\ Harmon.\ Anal.~40 (2016), pp.\ 321--351].
Tractability is achieved
by the use of non-linear algorithms, 
while linear algorithms cannot do the job.
\end{abstract}

\maketitle

We consider the problem of recovering 
%a periodic function on $[0,1]^d$ 
a high-dimensional function
$f\colon [0,1]^d \to \C$
from a class $F_d$ % \subset C([0,1]^d)$ 
using algorithms of the form
\begin{equation}\label{eq:alg}
 A_n \colon F_d \to L_2([0,1]^d), \quad A_n(f) = \phi(f(x_1),\hdots,f(x_n))
\end{equation}
with sampling points $x_j \in [0,1]^d$ and a recovery map $\phi\colon \C^n \to L_2$.
The error is measured in the $L_2$-norm and in a worst-case setting, i.e.,
\[
 \err(A_n,F_d,L_2) \,:=\, \sup_{f\in F_d} \Vert f - A_n(f) \Vert_2.
\]
It is known that this approximation problem suffers from the \emph{curse of dimensionality}
for most classical function classes $F_d$,
including the smoothness classes of $k$-times 
continuously differentiable functions \cite{HNUW,HNUW2}
or Sobolev spaces of mixed smoothness \cite{HW}.
That is, there exist $C,\gamma,\eps>0$
such that the %information-based 
(information) complexity
\begin{equation}\label{def:comp}
 n(\varepsilon,F_d,L_2) \,:=\, \min\big\{n\in\N \ |\ \exists A_n \colon \err(A_n,F_d,L_2) \le \eps\big\}
% \,\ge\, C (1+\gamma)^d
\end{equation}
is bounded below by $C (1+\gamma)^d$ for all $d\in\N$.
%The function $n(\eps,F_d)$ is called the (information-based) complexity of the problem.

A classical %and common 
approach 
to make this problem tractable in high dimensions is 
to consider weighted function classes $F_d$,
assuming a decreasing importance of the input variables.
%That is, one assumes a decreasing importance of the input variables
%in a predetermined order.
%
%to assume that a decreasing importance of the input variables
%is known a priori 
%and consider classes $F_d$ that
%assign decreasing weights to the different coordinates.
%%modeling a decreasing importance of the input variables.
This approach goes back to \cite{SW}
and has gained significant popularity, 
see also Remark~\ref{rem:weighted} about weighted Korobov spaces.
%see \cite{EP} for an overview and further references.
%
Unfortunately, 
it is much harder to provide reasonable unweighted function classes
where the approximation problem is tractable.
%the problem usually becomes much harder in the unweighted case. % for unweighted function classes $F_d$.
By \emph{unweighted}, we mean that all variables are of equal importance.
Formally, one may call a function class $F_d$ unweighted
if $f\circ \pi \in F_d$ for any $f\in F_d$ and any permutation matrix~$\pi$.

It was recently discovered by Goda \cite{God} that \emph{numerical integration}
is polynomially tractable on the unweighted class
\begin{equation}\label{eq:Fd1}
 F_d^{\log} \,:=\, \bigg\{ f \in C([0,1]^d) \ \Big|\ \sum_{k\in\Z^d} |\hat f(k)| \max(1,\log \Vert k \Vert_\infty) \le 1 \bigg\}.
\end{equation}
%of functions with absolutely convergent Fourier series.
Namely, the number of sampling points that are needed to approximate the integral of any function from $F_d^{\log}$ up to an error $\varepsilon$ is bounded above by a polynomial in $\varepsilon^{-1}$ and $d$.
%In this case, one may take $p=...$ and $q=...$.
Goda refines a method of Dick who already showed in \cite{Dick} that integration
is tractable on the class of $\alpha$-Hölder 
continuous functions with absolutely convergent Fourier series.
The class $F_d^{\log}$ avoids the Hölder condition 
by strengthening the condition of absolute convergence in a very slight manner.

Those results raise the question whether the problem of $L_2$-ap\-prox\-ima\-tion on $F_d^{\log}$,
whose complexity can only be higher than the complexity of the integration problem,
is polynomially tractable as well. 
We give a positive answer.

\begin{theorem}\label{thm:main}
There is a constant $c>0$ such that for all $d\in \N$ and %, all 
$\varepsilon\in (0,1/2)$, % and all $2\le p <\infty$, 
we have
\[
 n(\varepsilon,F_d^{\log},L_2) \,\le\, c\,d\,\varepsilon^{-3} \log^3(\varepsilon^{-1}).
 % c\,d\,p^3\, \varepsilon^{-3p/2} \log^3(\varepsilon^{-1}).
\]
\end{theorem}

More generally, the problem of sampling recovery on $F_d^{\log}$ in $L_p$ %([0,1]^d)$
is tractable for all $1\le p < \infty$, see Corollary~\ref{cor:Goda}.
In fact, we will show that 
$L_p$-approximation
is polynomially tractable for 
many classical unweighted classes of functions,
like Sobolev spaces of mixed smoothness
or Hölder continuous functions,
if it is additionally assumed that the Fourier series converges absolutely.
See Theorem~\ref{thm:general2} for a general result
and Corollaries~\ref{cor:Sob} and \ref{cor:Hoelder} 
for the corresponding examples. 
% sum of the Fourier coefficients is bounded by a constant.
Those findings are consequences of powerful results
from the theory of compressive sensing,
as will be discussed in Section~\ref{sec1}.

Interestingly, the curse of dimensionality is only mitigated due to the use of \textbf{non-linear} sampling algorithms. If one only allows linear algorithms, using linear recovery maps $\phi$ in \eqref{eq:alg}, 
the $L_2$-approximation problem on $F_d^{\log}$ suffers from the curse of dimensionality,
see Lemma~\ref{thm:lower}.
This fact also implies that the sampling recovery on $F_d^{\log}$ in $L_\infty$ %([0,1]^d)$ 
is intractable,
see Corollary~\ref{cor:uniform}.

%\begin{OP}
%We thus know that sampling recovery on $F_d^{\log}$ in $L_p([0,1]^d)$
%is tractable for all $p\in (0,2]$
%and intractable for $p=\infty$.
%Determine whether it is tractable for $p\in(2,\infty)$.
%\end{OP}

\begin{remark}
For numerical integration on the class $F_d^{\log}$, 
Goda \cite{God} obtained an upper bound of order $d^3\eps^{-3}$.
Since the upper bound from Theorem~\ref{thm:main} also holds for numerical integration,
we improve upon this bound if the dimension $d$ is larger than $\log^{3/2}(\eps^{-1})$. 
While the present paper was under review, 
Goda's upper bound on the complexity of integration 
has been further improved to $d\eps^{-3}$, see~\cite{CJ}.
\end{remark}

\begin{remark}The function class $F_d^{\log}$ defined in \eqref{eq:Fd1} 
is from the first version of the paper \cite{God} that appeared on arXiv in October 2022. 
In a later version, Goda replaced the term $\Vert k \Vert_\infty$
by $\min_{i \in {\rm supp}(k)} |k_i|$ 
and thus obtained tractability on an even larger class. 
I believe that this replacement is only possible for the integration problem 
and it will lead to intractability in the case of the approximation problem.
\end{remark}

\begin{remark} 
Other unweighted classes of functions 
where the approximation problem is tractable
are high-dimensional functions 
that show some ``low-dimensional structure'' like 
tensor products of univariate functions \cite{BDDG,KR,NR},
sums of low-dimensional functions \cite{WW},
or compositions of univariate functions
with linear functionals \cite{MUV}.
% ridge functions of the form $g(\langle\cdot,a\rangle)$
%with some $a\in\R^d$ and a univariate function $g$
%...
Moreover, it was proven in \cite{Xu}
that the approximation problem is weakly tractable
for the class of analytic functions defined on the cube with
directional derivatives of all orders bounded by 1.
For the integration problem,
we mention the star-discrepancy
as a further example \cite{HWWN}.
\end{remark}

\begin{remark}
Recently, the papers \cite{DT,JUV} also derived new bounds 
on the complexity of sampling recovery in $L_2$ using non-linear algorithms. % $\ell^1$-mini\-mi\-zation.
It is shown in \cite[Thm.\,3.2]{JUV}
that the $n$-th minimal error
is essentially bounded by the $s$-term
widths of the class $F_d$ in $L_\infty$
where it suffices to take $n$ quasi-linear in both $s$ and $d$.
The paper \cite{JUV} provides several examples of classical smoothness classes $F_d$
where non-linear sampling algorithms
have a better rate of convergence than any linear algorithm. %, see also \cite{DT}.
Here, we use instead a bound by
$s^{-1/2}$ times the $s$-term widths
of the class $F_d$ in the Wiener algebra,
see Lemma~\ref{lem:main} and Remark~\ref{rem:JUV2}.
The extra factor $s^{-1/2}$ eases the tractability analysis significantly.
It would be interesting to know whether \cite[Thm.\,3.2]{JUV} 
also leads to tractability results for the class $F_d^{\log}$ or related ones.
\end{remark}

\begin{remark}
It is remarkable that the complexity bounds in this paper
are not obtained with a sophisticated choice of interpolation nodes
but rather with i.i.d.\ uniformly distributed sampling points, see Lemma~\ref{lem:main}.
Recently, there has been much interest in the surprising
power of i.i.d.\ random information in comparison to optimal information.
I refer to the survey paper~\cite{SU}.
\end{remark}

\section{A general complexity bound}
\label{sec1}

Let $\mu$ be a probability measure on a set $D$
and let $\mathcal B = \{ b_k \mid k \in I \}$ 
be a countable orthonormal system in $L_2(D,\mu)$
such that 
\begin{equation}\label{eq:uni-bound}
 C_{\mathcal B} \,:=\, \sup_{k \in I} \Vert b_k \Vert_\infty \,<\, \infty.
\end{equation}
We say that $\mathcal B$ is a %uniformly 
bounded orthonormal system. 
%Let $F$ be a class of functions $f=\sum_{k=1}^\infty \hat f(k) b_k$ such that $\sum_{k=1}^\infty |\hat f (k)| \le 1$ for all $f\in F$.
Whenever $(c_k)_{k\in I}$ is an absolutely summable sequence of complex numbers,
the series $f = \sum_{k\in I} c_k b_k$ converges uniformly on $D$ and we have
\[
 \hat f (k) \,:=\, \int_D f(x) \overline{b_k(x)} \,{\rm d}\mu(x) \,=\, c_k, \quad k\in I.
\]
We consider the Wiener algebra
\begin{equation}\label{eq:defA}
 \mathcal A \,:=\, \mathcal A(\mathcal B) \,:=\, \left\{ f = \sum_{k\in I} c_k b_k \,\,\Big|\,\, \Vert f \Vert_{\mathcal A} := \sum_{k\in I} |c_k| < \infty \right\}
\end{equation}
%denote the Wiener algebra with respect to $\mathcal B$
and denote its unit ball by $\mathcal A_0(\mathcal B)$.
%equipped with the norm
%\[
% \Vert f \Vert_{\mathcal A} \,:=\, \sum_{k\in I} |\hat f (k)|.
%\]
For a finite index set $\Lambda\subset I$ and $f\in L_2$, we write
\[
 P_\Lambda(f) \,:=\, \sum_{k\in \Lambda} \hat f(k) b_k
\]
and denote by $\mathcal T(\Lambda)$ the span of the functions $b_k$ with $k\in \Lambda$.
Given a class $F \subset \mathcal A$,
we define the projection error
\[
 E_\Lambda^\infty(F) \,:=\, \sup_{f\in F} \Vert f - P_\Lambda f \Vert_\infty.
\]
We consider the following algorithm.

%Here,  $\Vert \cdot \Vert_{\ell^2(X)}$ denotes the %discrete $\ell^2$-semi-norm with respect to the point set $X$, i.e.,
%semi-norm given by
%\[
% \Vert f \Vert_{\ell^2(X)}^2 \,:=\, \frac{1}{\# X} \sum_{x\in X} |f(x)|^2.
%\]

\smallskip

\begin{alg}\label{alg}
Given a class $F\subset \mathcal A$, 
a finite index set $\Lambda\subset I$,
and a finite and non-empty point set $X \subset D$,
we consider the algorithm
$A_{F,\Lambda,X}$ 
which maps a function $f\in F$ to any solution of
\begin{equation}\label{eq:bp}
\begin{alignedat}{2}
 \min \Vert g \Vert_{\mathcal A} \quad &\text{subject to} \quad &&g\in \mathcal T(\Lambda)  \\
%\nolabel
&\text{and \ } && \frac{1}{\# X} \sum_{x\in X} |f(x) - g(x)|^2 \,\le\, E_\Lambda^\infty(F)^2.
\end{alignedat}
\end{equation}
\end{alg}

The algorithm is well-defined 
since the function $g = P_\Lambda(f)$
satisfies the constraints of the minimization problem.
Moreover, the output of the algorithm 
merely depends on $f$ via the data $(f(x))_{x\in X}$
and thus $A_{F,\Lambda,X}$ is a sampling algorithm
of the form \eqref{eq:alg}.

\smallskip

\begin{remark}
Algorithm~\ref{alg} is a direct translation 
of a classical algorithm from compressed sensing,
called basis pursuit denoising. Namely,
if we denote $x=(\hat g(k))_{k\in \Lambda}$,
$y=(f(x))_{x\in X}$ and $G=(b_k(x))_{k\in\Lambda, x\in X}$, 
solving \eqref{eq:bp} is equivalent to 
the $\ell^1$-minimization problem
%minimizing $\Vert x \Vert_1$
%subject to $\Vert Ax - y \Vert_2 \le \varepsilon$,
%where $y=(f(x))_{x\in X}$ and $A=(b_k(x))_{k\in\Lambda, x\in X}$.
\begin{equation}\label{eq:bp2}
 \min \Vert x \Vert_1 \quad \text{subject to} \quad \Vert Gx - y \Vert_2 \le E_\Lambda^\infty(F) \sqrt{\# X}.
\end{equation}
%where $y=(f(x))_{x\in X}$ and $A=(b_k(x))_{k\in\Lambda, x\in X}$.
%also known as basis pursuit.
The insight that $\ell^1$-minimization can be of great advantage over linear algorithms
goes back at least to \cite{CT,Don,GG}
and the problem of $\ell^2$-recovery on $\ell^1$-balls in $\R^m$. 
Different methods of computing a solution to \eqref{eq:bp2} can be found in \cite[Ch.\,15]{FR}. 
%for various methods of computing a solution to the latter. %\eqref{eq:bp2}.
\end{remark}

We now give an error bound for Algorithm~\ref{alg}.
The error is expressed in terms of
the best $s$-term widths of $F$ in $\mathcal A$, defined by
\[
 \sigma_s(F,\mathcal A) \,:=\, \sup_{f\in F}\, \inf_{\# \Lambda \le s\vphantom{|^|}}\, \sum_{k \in I \setminus \Lambda} |\hat f (k)|.
\]
We derive our error bound as
a simple implication of \cite[Thm.\,6.1]{RW}, 
where we reinterpret our data on $f$
as noisy data on $P_\Lambda f$.
The key ingredient to error bounds such as \cite[Thm.\,6.1]{RW}
is the restricted isometry property of the matrix $G=(b_k(x))_{k\in\Lambda, x\in X}$;
For $m$ random points, the matrix $m^{-1/2}G$
acts as a quasi-isometry on the set of vectors with at most $s$ non-zero entries,
where it suffices that $m$ is quasi-linear in $s$ and logarithmic in $\# \Lambda$,
see \cite[Thm.\,12.32]{FR} or \cite[Thm.\,5.2]{RW}. 
This property implies that vectors $c$ from the unit ball in $\ell^1(\Lambda)$
%the Fourier coefficients $c=(\hat f(k))_{k\in \Lambda}$ %, 
%which is contained in the unit ball of $\ell^1(\Lambda)$, 
can be recovered well from data of the form $y = Gc + e$ 
if the noise $e$ is sufficiently small, 
see \cite[Thm.\,6.12]{FR}.
In our case, we want to recover the vector $c=(\hat f(k))_{k\in \Lambda}$.
The data is given by $y=(f(x))_{x\in X}$
while $Gc =(P_\Lambda f(x))_{x\in X}$ and so
the entries of $e$ are bounded by $E_\Lambda^\infty(F)$.

\smallskip

\begin{lemma}[Compare {\cite[Thm.\,6.1]{RW}}]\label{lem:main}
There is a universal constant $c\ge 1$ such that the following holds.
Let $\mathcal B  = \{ b_k \mid k \in I \}$ be a bounded orthonormal system 
with respect to a probability measure $\mu$ and let
$F \subset \mathcal A(\mathcal B)$.
For $\gamma \in (0,1)$, $s\ge 2 C_{\mathcal B}^2$, 
and $\Lambda\subset I$, 
%put $c_\gamma = c \log(\gamma^{-1})$
%and $c_{\mathcal B}=c C_{\mathcal B}$
%and 
let $X_m$ be a set of
\[
 m \,\ge\, %C\, \log(\gamma^{-1})\, C_{\mathcal B}^2 \, \eps^{-p} \log^3(\eps^{-p}) \log(\# \Lambda)
 c\, \log(\gamma^{-1}) \, s \log^3(s) \log(\# \Lambda) %$n \ge c s \log^3(s) \log(\# I)$
\]
independent random points with distribution $\mu$.
Then, with probability at least $1-\gamma$, we have for all $2\le p \le \infty$
and $f\in F$ that
\[
 \Vert f - A_{F,\Lambda,X_m}(f) \Vert_p \,\le\, 
 c\,C_{\mathcal B}\, s^{-1/p} \sigma_s(F,\mathcal A) + c\, s^{1/2-1/p} E_\Lambda^\infty(F).
\]
%We can choose $c_\gamma = c \log(\gamma^{-1})$ with a universal constant $c>0$.
%
%In particular,
%\[
% n(C\varepsilon, F_d) % \hookrightarrow L_2, \Lambda^{\rm std}) 
% \,\le\, c_{1/2} \varepsilon^{-2} \log^3(\varepsilon^{-1}) \log(\#I).
%\]
\end{lemma}

\begin{proof}
We use \cite[Thm.\,6.1]{RW}
with $b_j$ instead of $\phi_j$ and %, $m=n$,
$\omega_j = C_{\mathcal B}$.
We choose $c_0$ as the maximum of all constants appearing in \cite[Thm.\,6.1]{RW}
and $c=c_0+1$.
%\[
% c^{-1} \,\ge\, \max\big\{ 2 ,\, 2 (c_2 + d_2)^{2/p} ((c_1  + 2d_2)C_{\mathcal B})^{(p-2)/p} \big\}
%\]
Then $X_m = \{ t_1, \hdots, t_m\}$ 
satisfies the conclusion of \cite[Thm.\,6.1]{RW} with probability $1-\gamma$.
Given $g\in F$, we apply this conclusion to the function $f=P_\Lambda g$
and samples $y_\ell=f(t_\ell) + \xi_\ell$ with
$\xi_\ell = g(t_\ell) - f(t_\ell)$ 
and hence $\Vert \xi\Vert_2 \le E_\Lambda^\infty(F) \sqrt m =: \eta$.
Then the function $f^\#$ defined in \cite[Thm.\,6.1]{RW} equals the function 
$A_{F,\Lambda,X_m}(g)$.
By the triangle inequality and Jensen's inequality, we have
\[
 \Vert g - A_{F,\Lambda,X_m}(g) \Vert_p
 \le E_\Lambda^\infty(F) + \Vert f - f^\# \Vert_p.
\]
As $\sigma_s(f)_{\omega,1}$ as defined in \cite{RW}
is bounded by $C_{\mathcal B} \sigma_s(F,\mathcal A)$,
we get
%$\sigma_s(f)_{\omega,1}/ \sqrt s \le \eps^{p/2}$.
%Thus,
\begin{align*}
 &\Vert f - f^\# \Vert_2 \,\le\, c_0 C_{\mathcal B}\, s^{-1/2} \sigma_s(F,\mathcal A) + c_0 E_\Lambda^\infty(F) \,=:\, R,\\
 &\Vert f - f^\# \Vert_\infty \,\le\, c_0 C_{\mathcal B}\, \sigma_s(F,\mathcal A) + c_0 s^{1/2} E_\Lambda^\infty(F) \,=\,R s^{1/2}.
\end{align*}
It only remains to apply $\Vert h \Vert_p^p \le \Vert h \Vert_2^2 \cdot \Vert h \Vert_\infty^{p-2}$
for $h=f - f^\#$.
\end{proof}

%In order to achieve an error of order $\eps$,
%we have to choose $s$ in Lemma~\ref{lem:main} % (and hence the number of samples)
%large enough for $C_{\mathcal B}\, s^{-1/p} \sigma_s(F,\mathcal A) \le \eps$ to hold.
%For classes $F$ in the unit ball of $\mathcal A$,
%we clearly have $\sigma_s(F,\mathcal A) \le 1$
%and may choose $s$ (and hence the number of samples) 
%of order $\eps^{-p}$.
%In fact, this is already enough for our tractability analysis.
%We thus get the following complexity bound in terms of the quantity

This leads to the following complexity bound.
Recall the definition of the (information) complexity
of sampling recovery in $L_2$
from \eqref{def:comp}
which is defined analogously for sampling recovery in $L_p$.
We also define the minimal size of an index set needed for a projection error $\eps>0$, i.e.,
\begin{equation}\label{eq:defN}
% N_1(\eps,F) \,&:=\, \min \big\{ s\in \N \mid s^{-1/2} \sigma_s(F,\mathcal A) \le C_{\mathcal B}^{-1} \eps\big\}, \\
 N_{\mathcal B}^\infty(\eps,F) \,:=\, \min \{ \# \Lambda \mid   E_\Lambda^\infty(F) \le \eps\}.
\end{equation}

\smallskip

%The complexity of sampling recovery in $L_p$
%with algorithms $A_n(f)=\phi(f(x_1),\dots,f(x_n))$
%and error demand $\eps>0$ is defined as
%%For error demand $\eps>0$ and some $1\le p \le \infty$,
%%we define the complexity
%\[
% n(\varepsilon,F,L_p) \,:=\, 
% \min\bigg\{n\in\N \,\,|\,  
% \inf_{\substack{x_1,\dots,x_n\in D\\ \phi\colon \C^n\to L_p}}\, 
% \sup_{f\in F}\, \Vert f - A_n(f) \Vert_p \,\le\, \eps\bigg\}.
%\]

\begin{prop}\label{thm:general}
There is a universal constant $C\ge 2$ such that,
for any bounded orthonormal system $\mathcal B$, any $F\subset \mathcal A_0(\mathcal B)$,
all $1 \le p < \infty$ and $\eps\in(0,1)$, % (2 C_{\mathcal B}^2)^{-1/p}$,
we have
\[
 n(\eps, F,L_p) % \hookrightarrow L_2, \Lambda^{\rm std}) 
 \,\le\, C\, \tilde\eps^{-r} \log^3(\tilde\eps^{-r})\, \log N_{\mathcal B}^\infty\big(\tilde\eps^{r/2},F\big)
% c \, N_1(\eps,F)\, \log^3(N_1(\eps,F))\, \log(N_{\mathcal B}^\infty(\eps,F)).
\]
with $\tilde\eps = \eps/ (CC_{\mathcal B})$ and $r=\max\{p,2\}$.
%\[
% n(c\,C_{\mathcal B}\,\eps, F,L_p) % \hookrightarrow L_2, \Lambda^{\rm std}) 
% \,\le\, C\, \eps^{-p} \log^3(\eps^{-p})\, \log N_{\mathcal B}^\infty(\eps^{p/2},F).
%% c \, N_1(\eps,F)\, \log^3(N_1(\eps,F))\, \log(N_{\mathcal B}^\infty(\eps,F)).
%\]
\end{prop}

\begin{proof}
The case $p<2$ follows from the case $p=2$,
so let $p\ge 2$.
Put $C=2c$ with $c$ from Lemma~\ref{lem:main}.
We fix $\gamma={\rm e}^{-1}$ and choose $\Lambda \subset I$ with $\# \Lambda \le N_{\mathcal B}^\infty(\tilde\eps^{p/2},F)$
and $E_\Lambda^\infty(F)\le \tilde\eps^{p/2}$.
Moreover, we put $s= \lceil \tilde\eps^{-p} \rceil$ % \min\{\lceil 2C_{\mathcal B}^2\rceil,\lceil \eps^{-p} \rceil\}$.
so that $s\ge 2 C_{\mathcal B}^2$.
Observing that
\[
 \sigma_s(F,\mathcal A) \,\le\, 1,
\]
Lemma~\ref{lem:main} yields the statement.
\end{proof}

%The complexity of recovery in $L_p$ with $p<2$
%is bounded above by the complexity for $p=2$.
Proposition~\ref{thm:general} leads to our main result on the tractability of $L_p$-approximation.
Here, given classes $F_d\subset\mathcal A_d$ for each $d\in\N$,
%where $d\in\N$ usually stands for the dimension of the domain,
we say that sampling recovery on $F_d$ in $L_p$ is polynomially tractable if
\[
\exists C,q,r\ge 0\colon \forall d\in\N \ \forall \eps>0 \colon \quad
 n(\varepsilon,F_d,L_p) \,\le\, C d^q \eps^{-r}.
\]
Recall the definition of $C_{\mathcal B}$, $\mathcal A_0(\mathcal B)$ and $N_{\mathcal B}^\infty(\eps,F)$ 
from \eqref{eq:uni-bound},
\eqref{eq:defA}
and \eqref{eq:defN}.

\smallskip

\begin{theorem}\label{thm:general2}
For every $d\in\N$, let
$\mathcal B_d$ be a bounded orthonormal system
and $F_d \subset \mathcal A_0(\mathcal B_d)$.
Assume that there are positive constants $c_1,c_2,\alpha,\beta,$ and $\gamma$
such that
\[
 C_{\mathcal B_d} \,\le\, c_1 d^\alpha
 \qquad\text{and}\qquad
 N_{\mathcal B}^\infty(\eps,F_d) \,\le\, \exp(c_2 d^\beta \eps^{-\gamma})
\]
for all $\eps>0$ and $d\in\N$.
%with 
%\[
%\exists c_1,\alpha\ge 0\colon \forall d\in\N \colon \quad
%% \sup_{b\in\mathcal B_d} \Vert b\Vert_\infty 
% C_{\mathcal B_d} \,\le\, c_1 d^\alpha
%\]
%%the uniform bound $C_{\mathcal B_d}$ from \eqref{eq:uni-bound} grows at most polynomial in $d$,
%and classes 
%%\[
%% F_d \,\subset\, \left\{ f \in \mathcal A(\mathcal B_d) \colon  \Vert f \Vert_{\mathcal A(\mathcal B_d)} \le 1 \right\},
%%\]
%$F_d \subset \mathcal A(\mathcal B_d)$,
%assume that 
%\[
%\exists c_2,\beta,\gamma\ge 0\colon \forall d\in\N \ \forall \eps>0 \colon \quad
% N_{\mathcal B}^\infty(\eps,F_d) \,\le\, \exp(c_2 d^\beta \eps^{-\gamma}),
%\]
%
%where $C_{\mathcal B_d}$ and $N_{\mathcal B}^\infty(\eps,F_d)$ are defined in ...
Then the problem of sampling recovery on $F_d$ in $L_p$
is polynomially tractable for all $1\le p<\infty$.
\end{theorem}

%\begin{proof}
%We only need to observe that 
%\[
% N_1(\varepsilon,F_d) \,\le\, \lceil c_1^2 d^{2\alpha} \varepsilon^{-2} \rceil .
%\]
%to obtain from Proposition~\ref{thm:general} that
%\[
% ...
%\]
%\end{proof}

%In fact, it suffices to choose $n \ge c s \log^3(s) \log(\# I)$,
%where $s$ is minimal such that $s^{-1/2} \sigma_s(F,\mathcal A) \le \varepsilon$.
%Here
%\[
% \sigma_s(F,\mathcal A)
%\]
%denotes the best $s$-term width of $F$ in $\mathcal A$.

That is, sampling recovery in $L_p$ on classes with absolutely convergent
(generalized) Fourier series is tractable
% we get tractability for $L_p$-approximation
even if the number of %(generalized) 
Fourier coefficients
needed for an $\varepsilon$-approximation of functions $f\in F_d$ in the uniform norm
grows super-exponentially in $d$ and $\eps^{-1}$,
as long as the growth is not double-exponential.

%Note that, if $F$ is a class of functions whose absolute sum of coefficients is bounded by one,
%we always have 
%\[
% N_1(\varepsilon,F) \,\le\, \lceil C_{\mathcal B}^2 \varepsilon^{-2} \rceil .
%\]
%Thus, assuming the uniform bound $C_{\mathcal B}$ is dimension-independent, 
%we obtain polynomial tractability for such a class %with absolutely summable Fourier series
%whenever the logarithm of $N_{\mathcal B}^\infty(\varepsilon,F)$ is bounded 
%by a polynomial in $\varepsilon^{-1}$ and the dimension $d$.
%This is a very weak condition.
%For instance, it is not even a problem if we need $(d\varepsilon^{-1})^d$ coefficients
%for an $\varepsilon$-approximation of functions $f\in F$ in the uniform norm.
%We will examine this further for the Fourier system in the next section.

\begin{remark}[Rate of convergence]\label{rem:JUV2}
For our tractability analysis, it was enough to observe that $\sigma_n(F,\mathcal A) \le 1$
for all $F\subset \mathcal A_0$.
In fact, it would not help for the classes considered in this paper
to take the decay of the widths $\sigma_n(F,\mathcal A)$ into account
as $n$ would have to be exponentially large in $d$ if we want widths significantly smaller than one.
This is different
if one is interested in studying the rate of convergence of the $n$th minimal error
\[
 \err(n,F_d,L_p) \,:=\, \inf_{\substack{x_1,\dots,x_n\in D\\ \phi\colon \C^n\to L_p}}\, 
 \sup_{f\in F}\, \Vert f - A_n(f) \Vert_p.
\]
%one should take the decay of $\sigma_n(F,\mathcal A) \le 1$
Often (e.g., for Sobolev classes of mixed smoothness) one
can choose the size of $\Lambda$ polynomial in $n$
in order to obtain from Lemma~\ref{lem:main} that
\[
 \err(C n \log^4 n, F,L_p) \,\lesssim\, n^{-1/p} \,\sigma_n(F,\mathcal A).
\]
In comparison, the paper \cite{JUV} recently revealed estimates of the form
\[
 \err(C n \log^4 n, F,L_2) \,\lesssim\, \sigma_n(F,\mathcal B)_{L_\infty}
\]
by means of the best $n$-term widths in $L_\infty$ instead of $\mathcal A$,
%which was recently discovered to hold for many examples in \cite{JUV}
where the power of the logarithm can be reduced to 3 for the trigonometric system.
%Note that the power of the logarithm in the oversampling factor can be reduced by one
%in case of the trigonometric system, see \cite[Theorem~3.2]{JUV}
%%Also note that, for the trignometric system,
%%the power of the logarithm in the oversampling factor can be reduced by one,
%%see \cite{HR}.
%The authors of \cite{JUV} used \eqref{...}
%to obtain new asymptotics for the sampling numbers of 
%a number of smoothness classes $F_d$
\end{remark}

\section{Results for specific classes} %{Tractability for the Fourier system}
\label{sec:UBFourier}

We now consider the Fourier system
\[
 \mathcal B_d \,:=\, \{ b_k = e^{2\pi i \langle k, \cdot \rangle} \mid k \in \Z^d\}
\]
that is orthonormal in $L_2([0,1]^d)$ with the Lebesgue measure 
and satisfies $C_{\mathcal B_d} = 1$.
We obtain the classical Wiener algebra $\mathcal A_d := \mathcal A(\mathcal B_d)$
of periodic functions with absolutely convergent Fourier series.
%with unit ball
%\[
% B_{\mathcal A_d} \,=\, \Big\{ f \in \mathcal A_d \mid \sum_{k\in \Z^d} |\hat f (k) | \,\le\, 1 \Big\}.
%\]
%
%\medskip
%
%
%\begin{corollary}
%For all $d\in \N$, let $F_d$ be a class of continuous functions on $[0,1]^d$
%whose Fourier coefficients have absolute sum at most one.
%If
%\[
%\exists C,p,q\ge 0\colon \forall d\in\N \ \forall \eps>0 \colon \quad
% N_{\mathcal B}^\infty(\eps,F_d) \,\le\, \exp(C d^p \eps^{-q})
%\]
%then $L_2$-approximation on $F_d$ with standard information is polynomially tractable.
%\end{corollary}
%
We consider three examples,
starting with the class $F_d^{\log}$ from the introduction. % \eqref{eq:Fd1}.
The following corollary also proves Theorem~\ref{thm:main}.

\begin{corollary}\label{cor:Goda}
Let $2\le p <\infty$.
There is a constant $c>0$ such that for all
$d\ge 2$ and $\varepsilon\in (0,1/2)$, we have
\[
 n(\varepsilon,F_d^{\log},L_p) \,\le\, c\,d\, \varepsilon^{-3p/2} \log^3(\varepsilon^{-p}).
\]
\end{corollary}

\begin{proof} %[Proof of Theorem~\ref{thm:main}]
For any $m\in\N$ and $f\in F_d^{\log}$, we have
\[
 \Vert f - P_{[-m,m]^d} f \Vert_\infty \,\le\, \log^{-1}(m+1)
\]
and thus
\[
 N_{\mathcal B_d}^\infty(\eps,F_d^{\log}) \,\le\, \exp(2d\eps^{-1}).
\]
Now the bound is obtained from Proposition~\ref{thm:general}.
%with $N_1(\eps,F^{s,d}_{\rm mix}) \le \lceil \eps^{-2} \rceil$.
\end{proof}

\vspace*{-1mm}

As a second example,
we consider Sobo\-lev spaces of mixed smoothness $s>1/2$
(also called unweighted Korobov spaces),
namely
\[
 H_{\rm mix}^{s,d} \,:=\, \bigg\{ f \in C([0,1]^d) \ \Big|\  \sum_{k\in\Z^d} |\hat f(k)|^2 
 \prod_{i=1}^d \max\{1,|k_i|^{2s}\} \le 1 \bigg\}.
\]
It is well known that $L_2$-approximation on $H_{\rm mix}^{s,d}$ suffers from the curse of dimensionality,
see, e.g., \cite{HW}.
The curse is relinquished in the presence of absolutely summable Fourier coefficients,
namely, for the class
\[
 F^{s,d}_{\rm mix} \,:=\, \bigg\{ f\in H_{\rm mix}^{s,d} \ \Big| \ \sum_{k\in\Z^d} |\hat f(k)| \le 1 \bigg\}.
\]

\begin{remark}[Weighted Korobov spaces]\label{rem:weighted}
A classical approach to relinquish 
the curse of dimensionality on $H_{\rm mix}^{s,d}$
is by introducing (product) weights $1\ge \gamma_1 \ge \gamma_2 \ge \hdots >0$
and
considering the class
\[
 H_{\rm mix}^{s,d,\gamma} \,:=\, \bigg\{ f \in C([0,1]^d) \mid  \sum_{k\in\Z^d} |\hat f(k)|^2 
 \prod_{i=1}^d \max\{1,\gamma_i^{-1} |k_i|^{2s}\} \le 1 \bigg\}.
\]
see, e.g., \cite{EP,KSS,NSW} and \cite[Ch.\,13]{DKP} and the references therein.
The sequence $\gamma$ models a decreasing importance of the variables.
%See, e.g., the book
It is known that $L_2$-approximation on $H_{\rm mix}^{s,d,\gamma}$
is polynomially tractable if the weights decay fast enough, namely, if
there is a constant $K>0$ with
\[
 \sum_{i=1}^d \gamma_i \,\le\, K \log(d), \quad d\ge 2,
% K_\gamma \,:=\, \sup_{d\in \N \setminus\{1\}}\, \frac{1}{\log d} \sum_{i=1}^d \gamma_i \,<\, \infty.
\]
see~\cite[Thm.\,1]{NSW}.
The present approach is more general:
For any $\gamma$ as above,
a simple computation shows the inclusion
\begin{equation}\label{eq:inclusion}
 H_{\rm mix}^{s,d,\gamma} \subset C\,d^q\, F^{s,d}_{\rm mix}
\end{equation}
with constants $C,q>0$ that only depend on $K$ and $s$.
\end{remark}

\begin{corollary}\label{cor:Sob}
For any $s>1/2$ and $1\le p<\infty$, 
the problem of $L_p$-approximation
is polynomially tractable on $F^{s,d}_{\rm mix}$.\\
More precisely, for $p\ge 2$,
there is a constant $C>0$,
depending only on $s$ and $p$,
such that for all $d\ge 2$ and $\varepsilon\in (0,1/2)$, we have
\[
 n(\varepsilon,F^{s,d}_{\rm mix},L_p) \,\le\, C\,d^2 \log(d)\, \varepsilon^{-p} \log^4(\varepsilon^{-1}).
\]
\end{corollary}

\begin{proof}
For $m\in\N$ and $f\in F^{s,d}_{\rm mix}$, we get from Hölder's inequality that
\begin{multline*}
 \Vert f - P_{[-m,m]^d} f \Vert_\infty^2 
 \,\le\, \bigg(\sum_{\Vert k \Vert_\infty > m} |\hat f(k)|\bigg)^2
 \,\le\, \sum_{\Vert k \Vert_\infty > m} \prod_{i=1}^d \max\{1,|k_i|\}^{-2s} \\
 \le\, d \cdot \sum_{\substack{k\in\Z^d \\ \vert k_1 \vert > m}} \prod_{i=1}^d \max\{1,|k_i|\}^{-2s}
  \,\le\, d \cdot c_s^{d-1} \sum_{\substack{k_1\in\Z \\ \vert k_1 \vert > m}} |k_1|^{-2s}
  \,\le\, d \cdot c_s^d \cdot m^{-2s+1}
\end{multline*}
with $c_s>0$ depending only on $s$,
and thus
\[
 N_{\mathcal B_d}^\infty(\eps,F^{s,d}_{\rm mix}) \,\le\, \big(d\, c_s^d\, \eps^{-2}\big)^{d/(2s-1)}.
\]
Now the bound is obtained from Proposition~\ref{thm:general}.
%with $N_1(\eps,F^{s,d}_{\rm mix}) \le \lceil \eps^{-2} \rceil$.
\end{proof}

\begin{remark}
In regard of \eqref{eq:inclusion}, we note that
\begin{equation}\label{eq:hom}
 n(\eps,r F_d, L_p) \,\le\, n(\eps/(2r), F_d, L_p)
\end{equation}
holds for any convex and symmetric class $F_d$
and any $r\ge 1$,
so that tractability on $F^{s,d}_{\rm mix}$ indeed
implies tractability on the classes $H_{\rm mix}^{s,d,\gamma}$.
Equation~\eqref{eq:hom} follows from the optimality of homogeneous algorithms
for linear problems, see \cite[Thm.\,1]{KK}.
\end{remark}

As a third example,
we consider the class of Hölder %continuous 
functions
that has been studied by Dick \cite{Dick} for numerical integration, namely,
\[
 F_d^\alpha \,:=\, \bigg\{ f\in C([0,1]^d) \ \Big|\ \sum_{k\in\Z^d} |\hat f(k)| \le 1 \, \land \, \sup_{x\ne y} \frac{|f(x)-f(y)|}{\Vert x - y\Vert_2^\alpha} \le 1 \bigg\}
\]
with some $0<\alpha \le 1$. Here, the supremum is taken over all $x,y\in\R^d$,
considering $f$ as a 1-periodic function on $\R^d$.

\begin{corollary}\label{cor:Hoelder}
For any $\alpha \in (0,1]$ and $1\le p<\infty$, 
the problem of $L_p$-approximation
is polynomially tractable on $F^{s,d}_{\rm mix}$.\\
More precisely, for $p\ge 2$,
there is a constant $c>0$, depending only on $\alpha$, 
such that for all $d\ge 2$ and $\varepsilon\in (0,1/2)$, we have
\[
 n(\varepsilon,F_d^\alpha,L_p) \,\le\, c\,d^2 \log^2(d)\, \varepsilon^{-p} \log^4(\varepsilon^{-1}).
\]
\end{corollary}

\begin{proof}
Let $S_m = P_{[-m,m]}$ be the univariate Fourier sum operator
and $C^\alpha_{\rm per}$ be the class of univariate $\alpha$-Hölder continuous functions with Hölder constant one.
It is well known that
\[
 \Vert S_m f \Vert_\infty \,\le\, c_1 \log(m) \Vert f \Vert_\infty
\]
for all $m \ge 2$ and all $f\in L_\infty([0,1])$ and
\[
 \Vert f - S_m f \Vert_\infty \,\le\, c_2 \log(m) m^{-\alpha}
\]
for all $m \ge 2$ and all $f\in C^\alpha_{\rm per}$,
% $f\in C^\alpha_{\rm per}([0,1])$,
where $c_1,c_2>0$ depend on nothing but $\alpha$,
see \cite{Jackson}.
Let now $f\in F_d^\alpha$.
For fixed $x_1,\dots,x_{i-1},x_{i+1},\dots,x_d \in [0,1]$,
the function $f$ as a function of $x_i$ is in $C^\alpha_{\rm per}$.
By the triangle inequality,
\begin{align*}
 \Vert f - P_{[-m,m]^d} f \Vert_\infty \,\le\,
 \sum_{i=1}^d \Vert S_m^{x_1} \cdots S_m^{x_{i-1}} (f-S_m^{x_i}f) \Vert_\infty \\
 \le\, \sum_{i=1}^d (c_1 \log(m) )^{i-1} \cdot c_2 \log(m) m^{-\alpha}
 \,\le\, \frac{(c_3 \log(m) )^d}{m^{\alpha}}.
\end{align*}
This term is smaller than $\eps$ for $m \ge \exp(c_4 d \log^2(d) \log(\eps^{-1}))$,
%BEMERKUNG: Sogar Exponent 1 im d-Logarithmus sollte reichen, aber dann muss man genauer pruefen
thus
\[
 \log N_{\mathcal B_d}^\infty(\eps,F_d^\alpha) \,\le\, c_5 d^2 \log^2(d) \log(\eps^{-1}).
\]
Now the bound is obtained from Proposition~\ref{thm:general}.
\end{proof}

It is interesting to note that the polynomial order $d^2 \eps^{-2}$
that we obtained for $L_2$-approximation
is the same as the one obtained in \cite{Dick} for numerical integration.
It is surely an interesting open problem to 
improve the complexity bounds in Corollaries~\ref{cor:Goda},
\ref{cor:Sob} and \ref{cor:Hoelder}
and/or provide matching lower bounds.

\section{Intractability with linear algorithms}

We only have tractability for the class $F_d^{\log}$ thanks to \textbf{non-linear} algorithms.
If we restrict to linear algorithms, we still have the curse of dimensionality. % for those classes.
This is implied by a result of \cite{Glu}, see also \cite[Lem.\,3.2]{Vyb}.
Namely, it is known that for any $m,n \in \N$, $m>n$, 
and any linear mapping $T_n\colon \C^m \to \C^m$ of rank $n$,
there is some $x\in \C^m$ with $\Vert x \Vert_1 \le 1$ and
\begin{equation}\label{eq:Glus}
\Vert x - T_nx \Vert_2 \,\ge\, \left( \frac{m-n}{m} \right)^{1/2}.
\end{equation}

\begin{lemma}\label{thm:lower}
The problem of $L_2$-approximation on the classes $F_d^{\log}$ with linear algorithms suffers from the curse of dimensionality. For any linear mapping $A_n \colon F_d^{\log} \to L_2$ with rank $n\le 5^d/2$ we have
\[
 \err(A_n,F_d^{\log},L_2) \,\ge\, \frac{1}{\sqrt 2}.
\]
\end{lemma}

Note that the curse of dimensionality even holds for the class of all linear algorithms,
using arbitrary linear information, instead of only sampling based algorithms.

\begin{proof}
Let $A_n\colon F_d^{\log} \to L_2$ be a linear operator of rank $n\le 5^d/2$
and let $\Lambda=\{-2,-1,0,1,2\}^d$ so that $\#\Lambda=5^d$.
%Moreover, let 
%\[
%\mathcal A_d^* \,:=\, \left\{ f \in \mathcal T(I) \colon \sum_{k\in I} |\hat f(k)| \le 1 \right\}
%\subset F_d^{\log}.
%\]
Consider the linear mapping $T\colon \C^\Lambda \to \mathcal T(\Lambda)$
that maps a vector of coefficients to the corresponding trigonometric polynomial.
%and the orthogonal projection $P\colon L_2 \to \mathcal T(\Lambda)$.
% canoncial isometry $T\colon \ell_1(I) \to \mathcal A_d^*$.
Then $T_n=T^{-1} P_\Lambda A_n T\colon \C^\Lambda \to \C^\Lambda$ 
is a linear mapping with rank at most $n$ %from $\C^{I}$ to $\C^n$
and \eqref{eq:Glus}
yields that there is some $x\in \C^\Lambda$ with $\Vert x \Vert_1 \le 1$
such that %$T_n x=0$ and
%By \cite{Glu}, see also \cite[Lemma~3.2]{Vyb}, 
\[
% \Vert Tx \Vert_2 \,=\, 
\Vert Tx - P_\Lambda A_n T x \Vert_2
\,=\, \Vert x - T_n x\Vert_2 \,\ge\, \left( \frac{5^d-5^d/2}{5^d} \right)^{1/2} \,=\, \frac{1}{\sqrt 2}.
\]
The function $f=Tx$ is contained in $F_d^{\log}$.
Moreover, $f\in \mathcal T(\Lambda)$ 
%while $A_n(f) \in \mathcal T(\Lambda)^\perp$
and thus
\begin{equation}\label{eq:lower2}
 \Vert f - A_n f \Vert_2 \,\ge\,
 \Vert f - P_\Lambda A_n f \Vert_2 \,\ge\,  \frac{1}{\sqrt 2},
\end{equation}
as claimed.
\end{proof}

Clearly, the lower bound \eqref{eq:lower2}
also holds for the uniform instead of the $L_2$-norm.
Since linear algorithms are optimal 
for the recovery problem in the uniform norm, see \cite{CW},
Lemma~\ref{thm:lower} 
implies that the problem of uniform recovery on the class $F_d^{\log}$ is intractable.
Namely,
the following holds for all algorithms of the form
\begin{equation}\label{eq:alg-gen}
 A_n \colon F_d^{\log} \to L_\infty([0,1]^d), \quad A_n(f) = \phi(L_1(f),\hdots,L_n(f))
\end{equation}
where $L_1,\hdots,L_n\colon F_d^{\log} \to \C$
are linear functionals, possibly chosen adaptively,
and $\phi\colon \C^n \to L_\infty$.
This includes the sampling algorithms of the form \eqref{eq:alg}.

\begin{corollary}\label{cor:uniform}
The problem of (sampling) recovery in the uniform norm 
on the classes $F_d^{\log}$ 
suffers from the curse of dimensionality. 
For any mapping $A_n \colon F_d^{\log} \to L_\infty$ of the form \eqref{eq:alg-gen} 
with $n\le 5^d/2$ we have
\[
 \sup_{f\in F_d^{\log}} \Vert f - A_n(f) \Vert_\infty \,\ge\, \frac{1}{\sqrt 2}.
\]
\end{corollary}

In a similar way, one can show that polynomial tractability 
cannot be achieved with linear algorithms
for the classes $F^{s,d}_{\rm mix}$ and $F_d^\alpha$ from Section~\ref{sec:UBFourier}.
We omit the details.

\medskip

\noindent
\textbf{Acknowledgement.}
This research was funded by the Austrian Science Fund (FWF) Project M~3212-N. 
For the purpose of open access, I have applied a CC BY public copyright licence to any author accepted manuscript version arising from this submission.
I thank Erich Novak, Mario Ullrich, 
Tino Ullrich and a referee for helpful comments.

\end{document}